\documentclass[12pt]{amsart}
\usepackage[margin=1.2in]{geometry} 
\usepackage{tikz-cd}
\usepackage{amsrefs}

\title{Idempotent Monads and Comonads Applied to Compactifications and Unitizations}
\author{Jeri Ann Spiker}
\address{School of Mathematical and Statistical Sciences
\\Arizona State University
\\Tempe, Arizona 85287}
\email{jaspiker@asu.edu}

\subjclass[2000]{46L05}

\keywords{monad, comonad, adjoint functors, category equivalence, reflective and coreflective subcategories, compactifications, unitizations, locally compact Hausdorff spaces, $C^*$-algebras}

\newcommand{\mc}[1]{\mathcal{#1}}

\newcommand{\tild}[1]{\widetilde{#1}} 
\newcommand{\ov}[1]{\overline{#1}}
\newcommand{\vphi}{\varphi}

\newcommand{\id}{\mathrm{id}} 
\newcommand{\inc}{\mathrm{Inc}} 

\newcommand{\C}{\mc{C}} 
\newcommand{\D}{\mc{D}}
\newcommand{\M}{\mc{M}} 
\newcommand{\N}{\mc{N}}  
\newcommand{\T}{\mc{T}} 

\newcommand{\F}{\operatorname{F}}
\newcommand{\G}{\operatorname{G}}

\newcommand{\spec}{\operatorname{Spec}} 
\newcommand{\functc}{\operatorname{C}}
\newcommand{\functn}{\operatorname{N}}
\newcommand{\neta}{\operatorname{N}^\eta}
\newcommand{\functm}{\operatorname{M}} 
\newcommand{\mpsi}{\operatorname{M}^\psi}
\newcommand{\ev}{\operatorname{ev}}
\newcommand{\functt}{\operatorname{T}}
\newcommand{\deltat}{\functt^\delta}
\newcommand{\functs}{\operatorname{S}}
\newcommand{\epsilons}{\functs^\epsilon}

\newcommand{\com}{\mathbf{Cpt}}
\newcommand{\ucom}{\mathbf{U_c}} 
\newcommand{\U}{\mathbf{U}}

\newtheorem{thm}{Theorem}[section] 
\numberwithin{equation}{section}

\newtheorem{cor}[thm]{Corollary}

\newtheorem{lemma}[thm]{Lemma}
\newtheorem{prop}[thm]{Proposition}

\theoremstyle{definition}
\newtheorem{definition}[thm]{Definition}
\newtheorem{notation}[thm]{Notation}

\theoremstyle{remark}
\newtheorem{remark}[thm]{Remark}

\title[An Application of Idempotent Monads and Comonads]{An Application of Idempotent Monads and Comonads to Compactifications and Unitizations}

\begin{document}

\baselineskip=15pt

\begin{abstract}
    This paper uses monads and comonads to establish a certain type of equivalence between two subcategories, one reflective and one coreflective, in a category whose objects represent compactifications of non-compact locally compact Hausdorff spaces. The equivalence is then examined in the dual category of unitizations of non-unital commutative $C^*$-algebras and subsequently generalized to the noncommutative case. 
\end{abstract}

\maketitle

\section{Introduction}

In this paper, we aim to use the techniques presented in~\cite{BKQ25} for applying idempotent monads and comonads to obtain a pair of subcategories, one reflective and one coreflective, that are equivalent in a particular way. 

Abstractly, we have a category $\C$, a reflective subcategory $\N$ with reflector $\neta: \C \to \N$, and a coreflective subcategory $\M$ with coreflector $\mpsi: \C \to \M$. When certain hypotheses are met, the adjunction $\neta \circ \inc_\M \dashv \mpsi \circ \inc_\N$ is an adjoint equivalence and we have the following naturally isomorphic functors $\neta \cong (\neta \circ \inc_\M) \circ \mpsi$ and $\mpsi \cong (\neta \circ \inc_\M) \circ \mpsi$. In~\cite{BKQ11}, the authors present an example of this equivalence in the context of coactions of groups on $C^*$-algebras, referring to pairs of equivalent subcategories satisfying this relationship as a pair satisfying the \textit{maximal-normal equivalence}. 

In~\cite{BKQ25}, the authors show that the existence of an idempotent monad and an idempotent comonad on $\C$ give rise to reflective and coreflective subcategories and they provide an efficient characterization of when the pair satisfy the maximal-normal equivalence. They present an example of this equivalence in the context of Fell bundles over \'etale groupoids. 

In this paper, we present an example of a pair of subcategories satisfying the maximal-normal equivalence in the context of compactifications of non-compact locally compact Hausdorff spaces and unitizations of non-unital $C^*$-algebras. In~\cite{JT22}, the authors view compactifications as a functorial process that takes objects in the category of locally compact Hausdorff spaces to objects in the category of compact Hausdorff spaces. In our paper, we instead view compactifications as single objects that are triples in one unifying category we call the category of compactifications of locally compact Hausdorff spaces. We do this for unitizations as well, viewing unitizations as triples in a unifying category called the category of unitizations of $C^*$-algebras. 

In Section~\ref{preliminaries}, we recall the necessary preliminaries presented in~\cite{BKQ25} regarding idempotent monads and comonads, their relationship with reflective and coreflective subcategories, and detail the hypothesis required for a pair of subcategories to satisfy the maximal-normal equivalence. We also present a technical lemma showing the maximal-normal equivalence is preserved under contravariant equivalence. The contravariance implies that a monad in one category induces a comonad on the other category.  This observation is relevant to our example since the category of compactifications of locally compact Hausdorff spaces and the category of unitizations of commutative $C^*$-algebras are contravariantly equivalent when equipped with the appropriate morphisms. 

In Section~\ref{compactifications}, we introduce the category of compactifications of locally compact Hausdorff spaces and denote it by $\com$. In this category, the usual one-point compactifications appear as final objects in certain fibers over this category. We use the one-point compactifications to construct an idempotent monad on $\com$ and the corresponding reflective subcategory contains all objects isomorphic to a one-point compactification. Dually, the usual Stone--\v{C}ech compactifications appear as initial objects in these same fibers and we use them to construct an idempotent comonad on $\com$. The corresponding coreflective subcategory contains all objects isomorphic to a Stone--\v{C}ech compactification. Further, we prove that these two subcategories satisfy the maximal-normal equivalence. 

Then, in Section~\ref{ucom}, we introduce the category of unitizations of commutative $C^*$-algebras and denote it by $\ucom$. We show that this category is contravariantly equivalent to $\com$ using Gelfand duality. Due to this equivalence and the aforementioned technical lemma, we obtain a maximal-normal equivalent pair in $\ucom$ automatically. The contravariance of the equivalence means that the idempotent monad on $\com$ induces an idempotent comonad on $\ucom$ and we show that the corresponding coreflective subcategory contains objects isomorphic to minimal unitizations. Dually, the idempotent comonad on $\com$ induces an idempotent monad on $\ucom$ and the corresponding reflective subcategory contains objects isomorphic to maximal unitizations. 

Finally, in Section~\ref{unitizations}, we construct idempotent monads and comonads on a more general category of unitizations without the commutativity assumption. We do this using maximal and minimal unitizations and show that the corresponding reflective and coreflective subcategories satisfy the maximal-normal equivalence. 

\section{Preliminaries}\label{preliminaries}

\cite{BKQ11} and \cite{BKQ25} give an overview of the general category theory we require for this paper. We summarize those results here. 
 
Let $\C$ be a category. Let $1_\C$ denote the identity functor on $\C$. If $\D$ is a subcategory of $\C$, then $\inc_\D: \D \to \C$ denotes the inclusion functor.

\begin{definition}
     Let $\functn: \C \to \C$ be a functor and $\eta: 1_\C \to \functn$ a natural transformation. If the natural transformations $\eta \functn: \functn \to \functn^2$ and $\functn\eta: \functn \to \functn^2$ are natural isomorphisms, we call the pair $(\functn, \eta)$ an \textit{idempotent monad on $\C$}.
\end{definition}

Suppose $(\functn, \eta)$ is an idempotent monad on $\C$. Let $\N$ be the full subcategory of $\C$ whose objects $x$ satisfy that $\eta_x: x \to \functn x$ is an isomorphism. Let $\neta: \C \to \N$ denote the functor obtained by regarding $\functn$ as a functor from $\C$ to $\N$ so that $\functn = \inc_\N \circ \neta$. 

\begin{lemma} \label{monad}
    Let $(\functn, \eta)$, $\N$, and $\neta$ be as given above. Then $\inc_\N: \N \to \C$ is left adjointable and $\neta$ is a left adjoint of $\inc_\N$. Thus, $\N$ is a reflective subcategory of $\C$. Moreover, $\eta$ is the unit of the adjunction $\neta \dashv \inc_\N$. 

\end{lemma}

\begin{proof}
    Let $x$ be an object in $\C$, $y$ an object in $ \N$, and $f:x \to y$ is a morphism in $\C$. Then $\eta_y$ is an isomorphism and so $\eta_y^{-1} \circ \functn f:\functn x \to y$ is the unique morphism completing the following diagram: \[\begin{tikzcd}
        x \arrow[swap]{d}{\eta_x} \arrow{dr}{f} \\
        \functn x \arrow[dashed]{r} \arrow[swap]{dr}{\functn f}& y\arrow{d}{\eta_y} \\
        & \functn y.
    \end{tikzcd}\]
    This means for for all $x$ in $\C$, $(\neta x, \eta_x)$ is a universal morphism from $x$ to $\N$. 

    It's well known that this choice of a universal morphism for each object gives rise to a reflective subcateogry. See \cite{BKQ11}, Section 2. 
\end{proof} 
Now, the dual case. 

\begin{definition}
    Let $\functm: \C \to \C$ be a functor and $\psi: \functm \to 1_\C$ a natural transformation. If the natural transformations $\psi \functm: \functm^2 \to \functm$ and $\functm\psi: \functm^2 \to \functm$ are natural isomorphisms, we call the pair $(\functm, \psi)$ an \textit{idempotent comonad on $\C$}. 
\end{definition}

Let $(\functm, \psi)$ be an idempotent comonad on $\C$. Let $\M$ be the full subcategory of $\C$ whose objects $x$ satisfy that $\psi_x$ is an isomoprhism. Let $\mpsi :\C \to \M$ denote the functor obtained by regarding $\functm$ as a functor from $\C$ to $\M$ so that $\functm = \inc_\M \circ \mpsi $. 

\begin{lemma}\label{comonad}
    Suppose $(\functm, \psi)$, $\M$, and $\mpsi $ are as given above. Then $\inc_\M: \M \to \C$ is right adjointable and $\mpsi $ is a right adjoint of $\inc_\M$. Thus, $\M$ is a coreflective subcategory of $\C$. Moreover, $\psi$ is the unit of the adjunction $\inc_\M \dashv \mpsi $. 
\end{lemma}

\begin{proof}
    Dual to Lemma~\ref{monad}, we have for all objects $y$ in $\C$, $(\functm y, \psi_y)$ is a universal morphism from $\M$ to $y$. In particular, if $x$ is an object in $\M$ and $f:x \to y$ then $\functm f \circ \psi_x^{-1}:x \to \functm y$ is the unique morphism completing the following diagram:

    \[\begin{tikzcd}
         \functm x \arrow[swap]{d}{\psi_x} \arrow{dr}{\functm f} \\
        x \arrow[dashed]{r} \arrow[swap]{dr}{f}& \functm y \arrow{d}{\psi_y} \\
        & y.
    \end{tikzcd}\]

    This gives rise to a coreflective subcategory. See~\cite{BKQ11}, Section 2. 
\end{proof}

\begin{lemma}\label{maxnor} 
    Suppose $(\functn , \eta)$, $\N$, $\neta$, $(\functm , \psi)$, $\M$, and $\mpsi $ are as given above. Suppose the natural transformations $\functn \psi: \functn \functm  \to \functn $ and $\functm \eta: \functm  \to \functm \functn $ are natural isomorphisms. Then we have the following results. 
    \begin{enumerate}
        \item[(i)] The adjunction $\neta \circ \inc_\M \dashv \mpsi  \circ \inc_\N$ is an adjoint equivalence between $\M$ and $\N$. 
        \item[(ii)] $\neta \cong (\neta \circ \inc_\M) \circ \mpsi $.
        \item[(iii)] $\mpsi  \cong (\mpsi  \circ \inc_\N) \circ \neta$.
    \end{enumerate}
\end{lemma}

\begin{proof}
    See \cite{BKQ25} Section 3.3. 
\end{proof}

We refer to any pair of reflective and coreflective subcategories satisfying the hypothesis in Lemma~\ref{maxnor} as a pair satisfying the \textit{maximal-normal} equivalence. This title was introduced in \cite{BKQ11} in the context of maximal and normal coactions of groups on $C^*$-algebras. 

Idempotent monads and idempotent comonads are preserved under category equivalence. In the case of a contravariant equivalence, which is relevant to this paper, we have the following technical lemma.

\begin{lemma}\label{induced}
    Suppose $\C$ is equivalent to $\D$ via contravariant functors $\F:\C \to \D$ and $\G: \D \to \C$. Let $\theta: 1_\D \to \F\G$ be a natural isomorphism given by the equivalence. Suppose $(\functn , \eta)$ is an idempotent monad on $\C$ and $(\functm , \psi)$ is an idempotent comonad on $\C$. 

    Let $\functt= \F\functn \G: \D \to \D$ and $\delta:\functt \to 1_\D$ be the natural transformation given by the assignment $\delta_x = \theta_x^{-1} \circ \F\eta_{\G x}$ for an object $x$ in $\D$.   

    Let $\functs = \F\functm \G: \D \to \D$ and $\epsilon:1_\D \to \functs$ be the natural transformation given by the assignment $\epsilon_x = \F\psi_{\G x} \circ \theta_x$. 

    Then, $(\functt, \delta)$ is an idempotent comonad on $\D$ and $(\functs, \epsilon)$ is an idempotent monad on $\D$. 

    Furthermore, if $\functn \psi$ is a natural isomorphism, then $\functt \epsilon$ is a natural isomorphism. If $\functm \eta$ is a natural isomorphism, then $\functs \delta$ is a natural isomorphism. 
\end{lemma}

\begin{proof}
First, we show that $(\functt, \delta)$ is an idempotent comonad on $\D$. It is obvious that $\functt$ is a functor and that $\delta$ is a natural transformation. So, $\functt \delta$ is a natural transformation. To see that it is a natural isomorphism, let $x$ be an object in $\D$. Then,
\begin{align*}
    \functt \delta_x &= \functt (\theta_x^{-1} \circ \F \eta_{\G x}) \\
    & = \functt(\theta_x^{-1} )\circ \functt(\F\eta_{\G x}) \\
    &  = \functt(\theta_x^{-1}) \circ \F (\functn \G\F(\eta_{\G x})).
\end{align*}

This is a composition of isomorphisms, therefore $\functt \delta_x$ is an isomorphism. 

Similarly, to see that $\delta T$ is a natural isomorphism, observe  
\begin{align*}
    \delta Tx & = \delta_{\F \functn \G x} \\
    & = \theta_{\F \functn \G x}^{-1} \circ F(\eta_{\G \F \functn \G x}). 
\end{align*}
This is once again a composition of isomorphisms. We then have that $(\functt, \delta)$ is an idempotent comonad on $\D$. 

To show $(\functs, \epsilon)$ is a idempotent monad is similar. 

Now, suppose $\functn \psi$ is a natural isomorphism. Then for an object $x$ in $\D$, 
$$\functt \epsilon_x = \F \functn \G \F \psi_{\G x} \circ \functt \theta_x.$$
This is a composition of isomorphisms, therefore an isomorphism. 

To show $\functs \delta$ is an isomorphism is similar. 

\end{proof}

\section{Compactifications}\label{compactifications}

Suppose $X$ is a non-compact locally compact Hausdorff space. A \textit{compactification} of $X$ is a pair $(K,\vphi)$ where $K$ is a compact Hausdorff space and $\vphi:X \to K$ is an embedding such that $\ov{\vphi(X)} = K$. 

 From now on, unless otherwise specified, we use locally compact Hausdorff space to refer to a space that is locally compact, Hausdorff, but not compact. We do this is because compactifying an already compact space is rarely interesting and introduces unnecessary pathologies.

We now present the basic facts about the one-point compactification and Stone--\v{C}ech compactification of a locally compact Hausdorff space $X$. 

\begin{definition}\label{onestar}  A compactification $(X^*,\iota)$ of a locally compact Hausdorff space $X$ is called a \textit{one-point compactification of $X$} if it has the following property. If $(K, \vphi)$ is another compactification of $X$, then there exists a unique surjective continuous map $\iota_K: K \to X^*$ such that
    \[ \begin{tikzcd}
    & K \arrow[dashed]{d}{\iota_K} \\
    X  \arrow[r,"\iota"] \arrow[ru,"\vphi"] & X^* 
    \end{tikzcd} \]
    commutes.
\end{definition}

\begin{lemma}
    If $X$ is a locally compact Hausdorff space, then there exists a one-point compactification of $X$. $(X^*, \iota)$ is unique in that if $(K, \vphi)$ is another compactification with the property given in Definition~\ref{onestar} then $\iota_K:K \to X^*$ is a homeomorphism. 
\end{lemma}

\begin{proof}
    \cite{JM}, page 185. 
\end{proof}

$(X^*, \iota)$ has the following useful extension property. 

\begin{lemma} \label{fstar}
    Suppose $X$ and $Y$ are non-compact locally compact Hausdorff spaces. Denote their one-point compactifications by $(X^*, \iota)$ and $(Y^*,\kappa)$ respectively. If $f:X \to Y$ is a continuous proper map, then there exists a unique continuous map $f^*:X^* \to Y^*$ such that
 \[ \begin{tikzcd}
        X \arrow[r,"\iota"] \arrow[d,"f"] & X^* \arrow[dashed]{d}{f^*} \\
        Y \arrow[r,"\kappa"] & Y^* 
    \end{tikzcd} \]
    commutes.
\end{lemma}

\begin{proof} 
   \cite{JM}, page 185. 
\end{proof}

\begin{definition}\label{betax}
    A compactification $(\beta X, \rho)$ of a locally compact Hausdorff space $X$ is called a \textit{Stone--\v{C}ech compactification of $X$} if it has the following property. If $f: X \to K$ is a continuous map and $K$ is a compact Hausdorff space then there exists a unique continuous map $\tild{f}:\beta X \to K$ such that 
    \[ \begin{tikzcd}
    & \beta X \arrow[dashed]{d}{\tild{f}}\\
    X \arrow[ru,"\rho"] \arrow[r,"f"] & K 
    \end{tikzcd} \]
    commutes. 

    Further, if $(K, \vphi)$ is a compactification of $X$, then $\tild{\vphi}$ is surjective.
\end{definition}

\begin{lemma}
    If $X$ is a locally compact Hausdorff space, then there exists a Stone--\v{C}ech compactification of $X$. $(\beta X, \rho)$ is unique in that if $(K, \vphi)$ is another compactification with the property given in Definition~\ref{betax} then $\tild{\vphi}:\beta X \to K$ is a homeomorphism. 
\end{lemma}

\begin{proof}
    \cite{JM}, page 241. 
\end{proof}

\begin{notation}\label{betaf}
    Suppose $X$ and $Y$ are locally compact Hausdorff spaces with Stone--\v{C}ech compactifications $(\beta X, \rho)$ and $(\beta Y, \sigma)$ respectively. Let $f:X \to Y$ be continuous. Then we use the usual notation $\beta f$ to denote the map $\tild{ \sigma \circ f}: \beta X \to \beta Y$.
\end{notation}

\begin{lemma}
    There is a category $\com$ whose objects are triples $(X, K, \vphi)$, where $X$ is a locally compact Hausdorff space and $(K,\vphi)$ is a compactification of $X$. A morphism is a pair $(f,F):(X,K,\vphi) \to (Y,L,\gamma)$ where $f:X \to Y$ is a continuous proper map and $F: K \to L$ is a continuous map such that
    \[ \begin{tikzcd}
    X \arrow{r}{\vphi} \arrow{d}{f} & K \arrow{d}{F} \\
    Y \arrow{r}{\gamma} & L 
\end{tikzcd} \]
commutes. 
\end{lemma}

\begin{proof}
    The identity morphism for an object $(X,K, \vphi)$ is $(\id_X, \id_K)$ and composition of morphisms is defined componentwise. This composition rule works because composition of (proper) continuous maps are (proper) continuous and if the top and bottom squares of the following diagram on the left commutes, then the diagram on the right commutes:

     \[ \begin{array}{c@{\quad\text{}\quad}c}
    \begin{tikzcd}
        X \arrow[r,"\vphi"] \arrow[d, "f"]& K \arrow[d,"F"] \\
        Y \arrow[r,"\gamma"] \arrow[d,"g"] & L \arrow[d,"G"]\\
        Z \arrow[r,"\tau"] & M 
    \end{tikzcd}
    &
    \begin{tikzcd}
        X \arrow[r,"\vphi"] \arrow[d, "g \circ f"] & K \arrow[d,"G \circ F"] \\
        Z \arrow[r, "\tau"] & M.
    \end{tikzcd}
    \end{array} \]
\end{proof}

\begin{remark}
    We ask that $f:X \to Y$ is proper so that we may consider extensions $f^*:X^* \to Y^*$ as given in Lemma~\ref{fstar} and so that the category $\com$ is contravariantly equivalent to the category of unitizations of commutative $C^*$-algebras that we define later.
\end{remark}

For a fixed locally compact Hausdorff space $X$, let $\com_X$ denote the subcategory of $\com$ where objects have first coordinate $X$ and morphisms have first coordinate $\id_X$. We call $\com_X$ the \textit{fiber of $\com$ over $X$}. 

$(X,X^*,\iota)$ is a final object in $\com_X$. If $(X,K,\vphi)$ is an object in the fiber $\com_X$ then Definition~\ref{onestar} makes it clear that $(\id_X, \iota_K)$ is the unique morphism in the fiber $\com_X$ from $(X,K,\vphi)$ to $(X,X^*,\iota)$. We call $(X, X^*, \iota)$ a \textit{one-point compactification triple. }

\begin{prop}\label{minimizer}
    There is a functor $\functn: \com \to \com$ given on objects by setting $\functn (X,K,\vphi)$ to be a final object in $\com_X$. That is, 
    $$\functn (X,K,\vphi) = (X,X^*, \iota).$$
    $\functn $ is given on morphisms $(f,F):(X,K,\vphi) \to (Y,L,\gamma)$ by 
    $$\functn (f,F) = (f,f^*):(X,X^*,\iota) \to (Y,Y^*,\kappa).$$
\end{prop}
\begin{proof}
    This is clear.
\end{proof}

\begin{prop}\label{eta}
    The assignment $\eta_{(X,K,\vphi)} = (\id_X, \iota_K)$ gives a natural transformation $\eta: 1_\com \to \functn $.
\end{prop}

\begin{proof}
    For all morphisms $(f,F):(X,K,\vphi) \to (Y,L,\gamma)$ we must check that the diagram
    \[\begin{tikzcd}[column sep=large]
        (X,K,\vphi) \arrow{r}{(\id_X,\iota_K)} \arrow[swap]{d}{(f,F)} & (X,X^*,\iota) \arrow{d}{(f,f^*)}\\
        (Y,L,\gamma)\arrow{r}{(\id_Y, \kappa_L)} & (Y, Y^*, \kappa). \\
    \end{tikzcd}\]
    commutes. 
    
    To check that a diagram commutes in our category amounts to checking it commutes in the first coordinate and in the second coordinate. That means checking that these two identities hold: 
    $$\id_Y \circ f = f \circ \id_X$$
    $$f^* \circ \iota_K = \kappa_L \circ F.$$
    
    The first identity is clear. For the second identity, let $\vphi(x) \in \vphi(X)$.  Then we have

    \begin{align*}
        f^* \circ \iota_K(\vphi(x)) & = f^* \circ \iota(x) & \text{by the commuting square for }(\id_X, \iota_K) \\
        & = \kappa \circ f(x) & \text{by the commuting square for } (f,f^*) \\
        & = \kappa_L \circ \gamma \circ f(x) & \text{by the commuting square for }(\id_Y, \kappa_L) \\
        & = \kappa_L \circ F (\vphi(x)) & \text{by the commuting square for }(f,F).
    \end{align*}

    Thus, $f^* \circ \iota_K$ and  $\kappa_L \circ F$ are continuous functions that agree on $\vphi(X)$, a dense subset of $K$, thus the identity holds on all of $K$ and the diagram commutes. 
\end{proof}

\begin{prop}
    The natural transformations $\functn \eta: \functn  \to \functn ^2$ and $\eta \functn : \functn  \to \functn ^2$ are natural isomorphisms.
\end{prop}

\begin{proof}
    First, we show that $\functn^2=\functn$. For any object $(X, K, \vphi)$ in $\com$ we have
    $$\functn ^2(X,K,\vphi) = \functn (X,X^*, \iota) = (X,X^*, \iota) = \functn (X,K,\vphi).$$

    For a morphism $(f,F):(X,K,\vphi) \to (Y,L,\gamma)$ we have $\functn^2(f,F):(X, X^*, \iota) \to (Y, Y^*, \kappa)$ where
    $$\functn ^2(f,F) = \functn (f,f^*) = (f,f^*) = \functn (f,F).$$

    Next, we consider $\functn  \eta: \functn \to \functn$. 
    $$\functn  \eta_{(X,K,\vphi)} = \functn (\id_X, \iota_K) = (\id_X, (\id_X)^*)=(\id_X,\id_{X^*}).$$

   The last equality follows from Lemma~\ref{fstar} as $(\id_X)^*:X^* \to X^*$ is the unique continuous map such that $(\id_X)^* \circ \iota = \iota$ so $(\id_X)^* =\id_{X^*}$. Thus, $\functn \eta: \functn  \to \functn $ is the identity transformation and therefore a natural isomorphism. 

   For $\eta \functn :\functn  \to \functn $ we have 
   $$\eta \functn (X,K,\vphi) = \eta_{(X,X^*,\iota)} = (\id_X, \iota_{X^*}) = (\id_X, \id_{X^*}).$$

   The last equality follows from Definition~\ref{onestar} as $\iota_{X^*}:X^* \to X^*$ is the unique continuous map such that $\iota_{X^*} \circ \iota = \iota$ so $\iota_{X^*} = \id_{X^*}$. Thus, $\eta \functn : \functn  \to \functn $ is the identity transformation and therefore a natural isomorphism. 
\end{proof}

\begin{definition}
    We denote by $\N$ the full subcategory of $\com$ whose objects $(X,K,\vphi)$ satisfy the property that $\eta_{(X,K,\vphi)}=(\id_X, \iota_K):(X,K,\vphi) \to (X,X^*, \iota)$ is an isomorphism. Let $\neta: \com \to \N$ be the functor obtained by regarding $\functn : \com \to \com$ as a functor from $\com$ to  $\N$ So, $\functn  = \inc_{\N} \circ \neta$.
\end{definition}

Notice, $(X, K, \vphi)$ is an object in $ \N$ if and only if $\iota_K:K \to X^*$ is a homeomorphism. So, we say $\N$ is the full subcategory of $\com$ containing all objects isomorphic to a one-point compactification triple.

We conclude the following from Lemma~\ref{monad}. 

\begin{thm} \label{comnthm}
    Let $\functn :\com \to \com$, $\eta: 1_\com \to \functn $, $\N$, and $\neta: \com \to \N$ be defined as above. Then: 
    \begin{enumerate}
        \item[(i)] $(\functn, \eta)$ is an idempotent monad on $\com$. 
        \item[(ii)] $\inc_{\N}$ is left adjointable, $\neta$ is a left adjoint of $\inc_\N$, and $\eta$ is the unit of the adjunction $\neta \dashv \inc_{\N}$. In particular, $\N$ is a reflective subcategory of $\com$. 
        \item[(iii)] For all objects $(X, K, \vphi)$ in $\com$, $(N(X,K,\vphi), \eta_{(X,K,\vphi)})$ is a universal morphism from $(X,K,\vphi)$ to $\N$. 
        
    \end{enumerate}
    Moreover, an object is in $\N$ if and only if it is isomorphic to a one-point compactification triple in $\com$. 
\end{thm}

\begin{remark}
    Item (3) from Theorem~\ref{comnthm} begs the following diagram. Let $(X, K, \vphi)$ be an object in $\com$, $(Y,L,\gamma)$ be an object in $\N$ and $(f,F):(X,K,\vphi) \to (Y,L,\gamma)$, then we have that $(f, \iota_L^{-1} \circ f^*)$ is the unique morphism completing the following diagram:

    \[\begin{tikzcd}
        (X,K,\vphi) \arrow[swap]{d}{(\id_X,\iota_K)} \arrow{rd}{(f,F)} \\
        (X,X^*, \iota) \arrow[swap]{dr}{(f,f^*)} \arrow[dashed,swap]{r}{} & (Y,L,\gamma) \arrow{d}{(\id_Y, \iota_L)} \\
        & (Y,Y^*,\kappa).
    \end{tikzcd}\]
\end{remark}

For a fixed locally compact Hausdorff space $X$, $(X, \beta X, \rho)$ is an initial object in $\com_X$. If $(X,K,\vphi)$ is an object in the fiber $\com_X$, then Definition~\ref{betax} makes it clear that $(\id_X, \tild{\vphi})$ is the unique morphism from $(X, \beta X, \rho)$ to $(X,K,\vphi)$ in the fiber $\com_X$. We call $(X, \beta X, \rho)$ a \textit{Stone--\v{C}ech compactification triple. } 

\begin{prop}
    There is a functor $\functm: \com \to \com$ given on objects by setting $\functm(X,K,\vphi)$ to be an initial object in $\com_X$. That is, 
    $$\functm(X,K,\vphi) = (X, \beta X,\rho).$$
    $\functm$ is given on morphisms $(f,F):(X,K, \vphi) \to (Y,L,\gamma)$ by 
    $$\functm(f,F) = (f, \beta f).$$
\end{prop}

\begin{proof}
    This is clear. 
\end{proof}

\begin{prop}
    The assignment $\psi_{(X,K,\vphi)}=(\id_X, \tild{\vphi})$ gives a natural transformation $\psi:\functm \to 1_\com$. 

    \begin{proof}
        We show that for all morphisms $(f,F):(X,K,\vphi) \to (Y,L,\gamma)$ the diagram 
        \[\begin{tikzcd}[column sep=large]
            (X,\beta X,\rho) \arrow[swap]{r}{(\id_X, \tild{\vphi})} \arrow{d}{(f,\beta f)} & (X,K,\vphi) \arrow{d}{(f,F)} \\
            (Y,\beta Y,\sigma) \arrow{r}{(\id_Y,\tild{\gamma})} & (Y,L, \gamma)
        \end{tikzcd}\]
        commutes. This is clear in the first coordinate. For the second we have
        \begin{align*}
            \tild{\gamma} \circ \beta f (\rho(x)) & = \tild{\gamma} \circ \sigma \circ f(x) & \text{by the commuting square for } (f,\beta f) \\
            &  = \gamma \circ f(x) & \text{by the commuting square for } (\id_Y, \tild{\gamma}) \\
            & = F \circ \vphi(x) &\text{by the commuting square for } (f,F) \\
            & = F \circ \tild{\vphi}(\rho(x)) & \text{by the commuting square for $(\id_X, \tild{\vphi})$},
        \end{align*}
        so $\tild{\gamma} \circ \beta f$ and $F \circ \tild{\vphi}$  are continuous and equal on $\rho(X)$, a dense subset of $\beta X$, therefore must be equal on all of $\beta X$. 
    \end{proof}
\end{prop}

\begin{prop}
    $\functm\psi:\functm^2 \to \functm $ and $\psi \functm : \functm ^2 \to \functm $ are natural isomorphisms. 
\end{prop}

\begin{proof}
    First, we show that $\functm ^2 = \functm $. For any object $(X, K, \vphi)$ we have
    $$\functm ^2(X,K,\vphi) = \functm (X, \beta X, \rho) = (X,\beta X, \rho) = \functm (X,K,\vphi).$$
    For a morphism $(f,F):(X,K, \vphi) \to (Y,L,\gamma)$ we have $M^2(f,F):(X, \beta X, \rho) \to (Y, \beta Y, \sigma)$ where
    $$\functm ^2(f,F) = \functm (f, \beta f) = (f, \beta f) = \functm (f,F).$$
    Next, we consider $\functm  \psi: \functm  \to \functm $. For an object $(X,K,\vphi)$ we have 
    $$\functm \psi_{(X,K,\vphi)} = \functm (\id_X, \tild{\vphi}) = (\id_X, \beta \id_X) = (\id_X, \id_{\beta X}).$$

    The last equality holds because $\beta \id_X = \tild{\rho\circ \id_X}=\tild{\rho}$ is the unique morphism such that $\rho =  \tild{\rho} \circ \rho$ so $\beta \id_X = \id_{\beta X}$. Thus, $\functm  \psi:\functm  \to \functm $ is the identity transformation and therefore a natural isomorphism. 

    For $\psi \functm :\functm  \to \functm $ we have 
    $$\psi \functm (X,K,\vphi) = \psi_{(X,\beta X, \rho)} = (\id_X, \tild{\rho}) = (\id_X, \id_{\beta X}).$$
    Thus, $\psi \functm : \functm  \to \functm $ is the identity transformation and therefore a natural isomorphism. 
\end{proof}

\begin{definition}
    We denote by $\M$ the full subcategory of $\com$ whose objects $(X,K,\vphi)$ satisfy the property that $\psi_{(X,K,\vphi)}=(\id_X, \tild{\vphi}): (X,\beta X, \rho) \to (X,K,\vphi)$ is an isomorphism.  Let $\mpsi:\com \to \M$ be the functor obtained by regarding $\functm : \com \to \com$ as a functor from $\com$ to $\M$. So, $\functm  = \inc_{\M} \circ \mpsi$. 
\end{definition}

Notice, $(X,K,\vphi)$ is an object in $\M$ if and only if $\tild{\vphi}:\beta X \to K$ is a homeomorphism. Thus, $\M$ is the full subcategory of all objects in $\com$ isomorphic to a Stone--\v{C}ech compactification triple. 

We conclude the following from Lemma~\ref{comonad}.

\begin{thm}~\label{commthm}    Let $\functm :\com \to \com$, $\psi:\functm  \to 1_\com$, $\M$, and $\mpsi: \com \to \M$ be defined as above. Then: 
    \begin{enumerate}
        \item[(i)] $(\functm, \psi)$ is an idempotent comonad on $\com$. 
        \item[(ii)] $\inc_\M$ is right adjointable, $\mpsi$ is a right adjoint of $\inc_\M$, and $\psi$ is the counit of the adjunction $\inc_\M \dashv \mpsi$. In particular, $\M$ is a coreflective subcategory of $\com$. 
        \item[(iii)] For all objects $(Y,K,\vphi)$ in $\com$, $(M(X,K,\vphi),\psi_{(X,K,\vphi)})$ is a universal morphism from $\M$ to $(X,K,\vphi)$. 
    \end{enumerate}

    Moreover, an object is in $\M$ if and only if it is isomorphic to a Stone--\v{C}ech compactification triple in $\com$. 
\end{thm}

\begin{remark}
    The diagram for item (3) in Theorem~\ref{commthm} follows. Let $(Y,L,\gamma)$ be an object in $\C$,  $(X,K,\vphi)$ an object in $\M$, and $(f,F):(X,K,\vphi) \to (Y,L,\gamma)$. Then $(f, \beta f \circ \tild{\vphi}^{-1})$ is the unique morphism completing the following diagram:

    \[\begin{tikzcd}
        (X,\beta X, \rho) \arrow[swap]{d}{(\id_X,\tild{\vphi})} \arrow{rd}{(f,\beta f)} \\
        (X,K,\vphi) \arrow[dashed]{r} \arrow[swap]{rd}{(f,F)} & (Y, \beta Y, \rho) \arrow{d}{(\id_Y,\tild{\gamma})} \\
        & (Y,L,\gamma).
    \end{tikzcd}\]
\end{remark}

  Now, we show that $\N$ and $\M$ satisfy the hypothesis given in Lemma~\ref{maxnor} and are therefore a pair satisfying the maximal-normal equivalence introduced in \cite{BKQ11}.

\begin{prop}\label{MetaNpsi}
    The natural transformations $\functm \eta: \functm  \to \functm \functn $ and $\functn \psi:\functn \functm  \to \functn $ are both isomorphisms. 
\end{prop}

\begin{proof}
    For an object $(X,K,\vphi)$ we have 
    \begin{align*}
        \functm \eta_{(X,K,\vphi)} = \functm (\id_X, \iota_K) = (\id_X, \beta \id_X)= (\id_X, \id_{\beta X}).
    \end{align*}

    So, $\functm \eta$ is a natural isomorphism. Similarly,

    \begin{align*}
        \functn \psi_{(X,K,\vphi)} =\functn (\id_X, \tild{\vphi})= (\id_X, (\id_X)^*) = (\id_X, \id_{X^*}).
    \end{align*}
    So, $\functn \psi$ is also a natural isomorphism. 
\end{proof}

As stated in Lemma~\ref{maxnor}, an immediate consequence of Proposition~\ref{MetaNpsi} is the following. 

\begin{thm}
    The categories $\N$ and $\M$ are equivalent. The adjunction $\neta \circ \inc_{\M} \dashv \mpsi \circ \inc_{\N}$ is an adjoint equivalence between $\N$ and $\M$. Moreover, we have 
    $$\neta \cong (\neta \circ \inc_{\M}) \circ \mpsi \text{ and } \mpsi \cong (\mpsi \circ \inc_{\N}) \circ \neta.$$
\end{thm}

\section{Unitizations of Commutative $C^*$-Algebras}\label{ucom}
In this section, we say that $\tau:A \to B$ is a homomorphism if $\tau$ is an adjoint preserving homomorphism between $C^*$-algebras $A$ and $B$. Similarly, any isomorphism between $C^*$-algebras is assumed to preserve adjoints.

Suppose $A$ is a $C^*$-algebra. A \textit{unitization} of $A$ is a pair $(B, \alpha)$ where $B$ is a unital $C^*$-algebra and $\alpha: A \to B$ is an injective homomorphism such that $\alpha(A)$ is an essential ideal in $B$.   

\begin{definition}\label{minunitization}
    A unitization $(\tild{A}, \iota)$ of a $C^*$-algebra $A$ is called a \textit{minimal unitization of $A$} if it has the following property. If $B$ is unital $C^*$-algebra and $\pi:A \to B$ is a homomorphism, then there exists a unique unital homomorphism $\tild{\pi}:\tild{A} \to B$ such that
    \[\begin{tikzcd}
        A \arrow{r}{\pi} \arrow[swap]{rd}{\iota} & B \\
        & \tild{A} \arrow[dashed,swap]{u}{\tild{\pi}}
    \end{tikzcd}\]
    commutes.
    
    Further, if $(B, \alpha)$ is a unitization of $A$, then $\tild{\alpha}$ is injective.
\end{definition}

\begin{lemma}
    If $A$ is a $C^*$-algebra, then there exists a minimal unitization of $A$. $(\tild{A}, \iota)$ is unique in that if $(B, \alpha)$ is another compactification with the property given in Definition~\ref{minunitization} then $\tild{\alpha}: \tild{A} \to B$ is an isomorphism. 
\end{lemma}

\begin{proof}
    \cite{GM}, page 12 and 39. 
\end{proof}

\begin{definition}\label{multiplieralgebra}
    A unitization $(M(A), \rho)$ of a $C^*$-algebra $A$ is called a \textit{maximal unitization of $A$} if it has the following property. If $(B,\alpha)$ is some other unitization of $A$, then there exists a unique injective homomorphism $\ov{\alpha}:B \to M(A)$ such that 
\[\begin{tikzcd}
        & M(A) \\
        A \arrow{ru}{\rho} \arrow{r}{\alpha} & B \arrow[dashed,swap]{u}{\ov{\alpha}}
    \end{tikzcd}\]
    commutes.
\end{definition}

\begin{lemma}
    If $A$ is a $C^*$-algebra, then there exists a maximal unitization of $A$. $(M(A), \rho)$ is unique in that if $(B,\alpha)$ is another compactification with the property given in Definition~\ref{multiplieralgebra} then $\ov{\alpha}:B \to M(A)$ is an isomorphism. 
\end{lemma}

\begin{proof}
    \cite{RW}, page 26. 
\end{proof}

$(M(A), \rho)$ has the following useful extension property. 

\begin{lemma}\label{mtau}
    Suppose $A$ and $B$ are $C^*$-algebras. Let $(M(A), \rho)$ and $(M(B),\sigma)$ denote their maximal unitizations. If $\tau:A \to B$ is a nondegenerate homomorphism, then there exists a unique unital homomorphism $M\tau:M(A) \to M(B)$ such that the diagram

    \[\begin{tikzcd}
        B \arrow{r}{\sigma} & M(B) \\
        A \arrow{u}{\tau} \arrow{r}{\rho} & M(A) \arrow[dashed,swap]{u}{M\tau}
    \end{tikzcd}\]
    commutes. 
\end{lemma}

\begin{proof}
    \cite{RW}, page 27. 
\end{proof}

\begin{lemma}\label{lemucom}
    There is a category $\ucom$ whose objects are triples $(A, B, \alpha)$ where $A$ is a non-unital commutative $C^*$-algebra and $(B, \alpha)$ is a unitization of $A$ where $B$ is commutative. A morphism is a pair $(\tau, \pi): (A, B, \alpha) \to (C,D, \beta)$ where $\tau: A \to C$ is a nondegenerate homomorphism and  $\pi:B \to D$ is a unital homomorphism such that the diagram
    \[\begin{tikzcd}
        C \arrow{r}{\beta} & D \\
        A \arrow{u}{\tau} \arrow{r}{\alpha} & B \arrow{u}{\pi}
    \end{tikzcd}\]
    commutes.
\end{lemma}

\begin{proof}
    The identity morphism for an object $(A, B, \alpha)$ is $(\id_A, \id_B)$ and composition of morphisms is defined componentwise. This composition rule works because a composition of nondegenerate (resp. unital) homomorphisms is a nondegenerate (resp. unital) homomorphism and, in the following diagram, if the top and bottom squares in the left diagram commute, then the right diagram commutes:

    \[ \begin{array}{c@{\quad\text{ }\quad}c}
    \begin{tikzcd}
    E \arrow[r, "\gamma"] & F \\
    C \arrow[u, "\vphi"]  \arrow[r,"\beta"] & D \arrow[u,"\phi"]\\
    A \arrow[u,"\tau"] \arrow[r,"\alpha"] & B\arrow[u, "\pi"]
    \end{tikzcd}
    &
    \begin{tikzcd}
    E \arrow[r,"\gamma"] & F\\
    A \arrow[u,"\vphi \circ \tau"] \arrow[r,"\alpha"] & B.  \arrow[u,swap,"\phi \circ \pi"] 
    \end{tikzcd}
    \end{array}\]
\end{proof}

\begin{remark}
    We ask that $A$ is non-unital and that $\tau:A \to C$ is nondegenerate so that $\ucom$ is contravariantly equivalent to $\com$. Requiring $\tau:A \to C$ to be nondegenerate also allows us to consider the extension $M\tau: M(A) \to M(C)$ as given in Lemma~\ref{mtau} where $(M(A), \rho)$ and $(M(C), \sigma)$ are maximal unitizations of $A$ and $B$, respectively.
\end{remark}

We'd now like to show that $\com$ is contravariantly equivalent to $\ucom$ using Gelfand duality.

For a locally compact Hausdorff space $X$, let $C_0(X)$ denote the set of continuous functions that vanish at infinity on $X$ with the usual structure turning it into a $C^*$-algebra. For a continuous proper map $f:X \to Y$ between locally compact Hausdorff spaces, let $\functc f: C_0(Y) \to C_0(X)$ denote the nondegenerate homomorphism $\functc f(g) = g \circ f$.

Similarly, for a compact Hausdorff space $K$, let $C(K)$ denote the $C^*$-algebra of continuous functions on $K$, and if $F: K \to L$ is a continuous map between compact Hausdorff spaces, let $\functc F:C(L) \to C(K)$ be the unital homomorphism defined by $\functc (F)(G) = G \circ F$. 

To develop a functor from $\com$ to $\ucom$, we'd like to use the following lemma mentioned in \cite{RW}, page 24. 

\begin{lemma}\label{vphistar}
    Suppose $(X,K,\vphi)$ is an object in $\com$. Define $\vphi_*: C_0(X) \to C(K)$ by 
    $$\vphi_*(f)(z) = \begin{cases}
        f(x) \text{ if } z= \vphi(x) \in \vphi(X) \\ 
        0 \text{ if } z \notin \vphi(X)
    \end{cases}.$$
    Then $(C(K),\vphi_*)$ is a unitization of $C_0(X)$. 
\end{lemma}

\begin{prop}
    There is a contravariant functor $\functc:\com \to \ucom$ that acts on objects by 
    $$\functc(X,K,\vphi) = (C_0(X),C(K),\vphi_*)$$
    and on morphisms $(f,F):(X,K,\vphi) \to (Y,L,\gamma)$ by 
    $$\functc(f,F) = (\functc f, \functc F):(C_0(Y), C(L), \gamma_*) \to (C_0(X),C(K),\vphi_*).$$  
\end{prop}

\begin{proof}
    Lemma~\ref{vphistar} confirms that $\functc{(X,K,\vphi)}$ is an object in $\ucom$ and straightforward calculations show the diagram
    \[\begin{tikzcd}
    C_0(X) \arrow{r}{\vphi_*} & C(K) \\
    C_0(Y) \arrow{u}{\functc f} \arrow{r}{\gamma_*} & C(L) \arrow{u}{\functc F}.
\end{tikzcd}\]
commutes. 
\end{proof}

Now, we establish a functor from $\ucom$ to $\com$. For a commutative $C^*$ algebra $A$, let $\widehat{A}$ denote the locally compact Hausdorff space of characters on $A$. For a nondegenerate homomorphism $\tau:A \to B$ between commutative $C^*$-algebras, let $\hat{\tau}: \widehat{B} \to \widehat{A}$ denote the continuous proper map $\hat{\tau}(\gamma) = \gamma \circ \tau$. 

It is a routine exercise to prove the following lemma. 

\begin{lemma}\label{utoc}
    Suppose $(A, B,\alpha)$ is an object in $\ucom$. Then there exists an embedding $\phi_\alpha: \widehat{A} \to \widehat{B}$ with dense range. That is, there is a map $\phi_\alpha:\widehat{A} \to \widehat{B}$ so that $(\widehat{B}, \phi_\alpha)$ is a compactification of $\widehat{A}$. 
\end{lemma}

\begin{prop}
    There is a contravariant functor $\spec: \ucom \to \com$ that acts on objects by 
    $$\spec(A, B, \alpha) = (\widehat{A}, \widehat{B}, \phi_\alpha)$$ 
    and on morphisms $(\tau, \pi):(A, B, \alpha) \to (C, D, \beta)$ by 
    $$\spec(\tau, \pi) = (\hat{\tau} , \hat{\pi} ):(\widehat{C}, \widehat{D}, \phi_\beta) \to (\widehat{A}, \widehat{B}, \phi_\alpha)$$
\end{prop}

\begin{proof}
Lemma~\ref{utoc} confirms that $\spec(A, B,\alpha)$ is an object in $\com$ and straightforward calculations show the diagram
\[\begin{tikzcd}
    \widehat{C} \arrow{r}{\phi_\beta}\arrow{d}{\hat{\tau} } & \widehat{D} \arrow{d}{\hat{\pi} } \\
    \widehat{A} \arrow{r}{\phi_\alpha} & \widehat{B}.
\end{tikzcd}\]
commutes.
\end{proof}

For a locally compact (or compact) Hausdorff space $X$, let $\ev_X:X \to \widehat{C_0(X)}$ denote the evaluation homeomorphism. That is, $\ev_X(x)(f) = f(x)$. For a commutative $C^*$-algebra $A$, let $\Gamma_A: A \to C_0(\widehat{A})$ denote the Gelfand transform. That is, $\Gamma_A(a)(\gamma) = \gamma(a)$. 

\begin{prop}
    $\spec\functc: \com \to \com$ is naturally isomorphic to $1_\com$ via the natural isomorphism $\ev:1_\com \to \spec \functc$ given by $\ev_{(X,K,\vphi)} = (\ev_X, \ev_K)$.

    $\functc \spec: \ucom \to \ucom$ is naturally isomorphic to $1_\ucom$ via the natural isomorphism $\Gamma:1_{\ucom} \to \functc \spec$ given by $\Gamma_{(A,B,\alpha)} = (\Gamma_A, \Gamma_B)$. 
    
    \end{prop}

\begin{proof} 

For all objects $(X,K,\vphi)$ in $\com$, $(\ev_X, \ev_K):(X,K,\vphi) \to (\widehat{C_0(X)}, \widehat{C(K)}, \phi_{\vphi_*})$ is an isomorphism in $\com$ because the diagram
\[\begin{tikzcd}
        X \arrow{r}{\vphi} \arrow{d}{\ev_X} & K \arrow{d}{\ev_K} \\
        \widehat{C_0(X)}\arrow{r}{\phi_{\vphi_*}} & \widehat{C(K)}.
    \end{tikzcd}\]
    commutes and each coordinate is a homeomorphism. Moreover, for all morphisms $(f,F): (X,K,\vphi) \to (Y,L,\gamma)$, straighforward calculations show that the diagram 
     \[\begin{tikzcd}[column sep=large]
    (X,K,\vphi) \arrow{r}{(\epsilon_X, \epsilon_K)} \arrow{d}{(f,F)} & (\widehat{C_0(X)}, \widehat{C(K)}, \phi_{\vphi_*}) \arrow{d}{(\widehat{ \functc f}, \widehat{ \functc F})} \\
    (Y, L, \gamma) \arrow{r}{(\epsilon_Y, \epsilon_L)} & (\widehat{C_0(Y)}, \widehat{C(L)}, \phi_{\gamma_*})
    \end{tikzcd}\]
    commutes. So, $\ev:1_\com \to \spec \functc$ is a natural isomorphism. 

    Similarly, for all objects $(A, B,\alpha)$ in $\ucom$, $(\Gamma_A, \Gamma_B)$ is an isomorphism because the diagram 
\[ \begin{tikzcd}
        C_0(\widehat{A}) \arrow{r}{(\phi_\alpha)_*} & C(\widehat{B}) \\
        A \arrow{u}{\Gamma_A} \arrow{r}{\alpha} & B \arrow{u}{\Gamma_B}
    \end{tikzcd}\]
    commutes and each coordinate is an isomorphism. Morevoer, for all morphisms $(\tau, \pi):(A,B,\alpha) \to (C,D,\beta)$, the diagram 
    \[\begin{tikzcd}
        (C,D,\beta) \arrow{r}{(\Gamma_C,\Gamma_D)} & (C_0(\widehat{C}),C(\widehat{D}), \functc \phi_\beta)\\\
        (A,B, \alpha) \arrow{u}{(\tau,\pi)} \arrow{r}{(\Gamma_A, \Gamma_B)} & (C_0(\widehat{A}), C(\widehat{B}), \functc \phi_\alpha) \arrow[swap]{u}{(\functc \hat{ \tau}, \functc \hat{ \pi})}
    \end{tikzcd}\]
 commutes so $\Gamma: 1_{\ucom} \to \functc \spec$ is a natural isomorphism. 
\end{proof}

\begin{cor}\label{comucom}
    $\com$ and $\ucom$ are contravariant equivalent categories. 
\end{cor}

We immediately apply Lemma~\ref{induced} to $(\functn , \eta)$ and $(\functm,\psi)$ to induce an idempotent comonad and monad on $\ucom$.

First, let $(\functt, \delta)$ denote the idempotent comonad on $\ucom$ induced by $(\functn, \eta)$. As usual, let $\T$ denote the full subcategory whose objects $(A,B,\alpha)$ satisfy that $\delta_{(A,B,\alpha)}$ is an isomorphism. Let $\deltat$ be the functor obtained by regarding $\functt$ as a functor from $\ucom$ to $\T$ so that $\functt = \inc_\T \circ \deltat$. 

If $(A, B, \alpha)$ is an object in $\ucom$ and $(\widehat{A}, \widehat{A}^*, \iota)$ denotes the one-point compactification triple associated to $\widehat{A}$, then $$T(A, B, \alpha) = (C_0(\widehat{A}), C(\widehat{A}^*), \iota_*).$$ $(C(\widehat{A}^*), \iota_*)$ is a minimal unitization of $C_0(\widehat{A})$ since it has the property given in Definition~\ref{minunitization}, so $T(A,B,\alpha)$ is isomorphic to the minimal unitization triple associated to $C_0(\widehat{A})$. 

Conversely, if $(A, \tild{A}, \kappa)$ is a minimal unitization triple, then $(\widehat{\tild{A}}, \phi_\kappa)$ is a one-point compactification of $\widehat{A}$ as it has the property given in Definition~\ref{onestar}. Thus, $\iota_{\widehat{\tild{A}}}:\hat{\tild{A}} \to \hat{A}^*$ is a homeomorphism which implies
$$\delta_{(A, \tild{A}, \kappa)} = (\Gamma_A^{-1}, \Gamma_{\tild{A}}^{-1} \circ \functc \iota_{\widehat{\tild{A}}})$$
is an isomorphism. So, $(A, \tild{A}, \kappa)$ is an object in $ \T$. Therefore, an object in $\ucom$ is an object in $\T$ if and only if it is isomorphic to a minimal unitization triple. 

We summarize our findings below.

\begin{cor}
    The idempotent monad $(\functn, \eta)$ on $\com$ induces an idempotent comonad $(\functt, \delta)$ on $\ucom$. With $\T$ and $\deltat$ as given above, we have that $\inc_\T$ is right adjointable, $\deltat$ is a right adjoint of $\inc_\T$, and $\delta$ is the counit of the adjunction $\inc_\T \dashv \deltat$. 

    Moreover, an object $(A, B, \alpha)$ is in $\T$ if and only if $(A, B, \alpha)$ is isomorphic to a minimal unitization triple. 
\end{cor}

Now, let $(\functs, \epsilon)$ denote the idempotent monad on $\ucom$ induced by $(\functm, \psi)$ . Let $\mc{S}$ denote the full subcategory whose objects satisfy that $\epsilon_{(A, B, \alpha)}$ is an isomorphism. Let $\epsilons$ be the functor obtained by regarding $\functs$ as a functor from $\ucom$ to $\mc{S}$ so that $\functs = \inc_{\mc{S}} \circ \epsilons$. 

If $(A, B, \alpha)$ is an object in $\ucom$ and $(\widehat{A}, \beta \widehat{A}, \rho)$ denotes the Stone--\v{C}ech compactification triple associated to $\widehat{A}$, then 
$$\functs(A, B, \alpha) = (C_0(\widehat{A}), C(\beta\widehat{A}),\rho_*).$$
 $(C(\beta \widehat{A}), \rho_*)$ is a maximal unitization of $C_0(\widehat{A})$ since it has the property given in Definition~\ref{multiplieralgebra}. So, $\functs(A, B, \alpha)$ is isomorphic to the maximal unitization triple associated to $C_0(\widehat{A})$. 

Conversely, if $(A, M(A), \sigma)$ is a maximal unitization triple, then $(\widehat{M(A)}, \phi_\sigma)$ is a compactification of $\widehat{A}$ with the property given in Definition~\ref{betax} and thus $\tild{\phi_\sigma}:\beta \widehat{A} \to \widehat{M(A)}$ is a homeomorphism which implies 
$$\epsilon_{(A,M(A),\sigma)} = (\Gamma_A, \functc{\tild{\phi_\sigma}} \circ  \Gamma_{M(A)})$$
is an isomorphism so $(A, M(A), \sigma)$  is an object in $\mc{S}$. Thus, an object in $\ucom$ is in $\mc{S}$ if and only if it is isomorphic to a maximal unitization triple. 

\begin{cor}
    The idempotent comonad $(\functm, \psi)$ on $\ucom$ induces an idempotent monad $(\functs, \epsilon)$ on $\com$. With $\mc{S}$ and $\epsilons$ given as above, we have that $\inc_\mc{S}$ is left adjointable, $\epsilons$ is a left adjoint of $\inc_\mc{S}$, and $\epsilon$ is the unit of the adjunction $\epsilons \dashv \inc_\mc{S}$. 

    Moreover, $(A, B, \alpha)$ is an object in $\mc{S}$ if and only if $(A, B, \alpha)$ is isomorphic to a maximal unitization triple. 
\end{cor}

Finally, by Lemma~\ref{induced}, we conclude that since $(N, \eta)$ and $(M,\psi)$ satisfy the maximal-normal equivalence, $(T, \delta)$ and $(S, \epsilon)$ do too. 

\begin{cor}
    The categories $\T$ and $\mc{S}$ are equivalent. The adjunction $\epsilons \circ \inc_{\T} \dashv \deltat \circ \inc_{\mc{S}}$ is an adjoint equivalence between $\T$ and $\mc{S}$. Moreover, we have
    $$\epsilons \cong (\epsilons \circ \inc_{\T}) \circ \deltat \text{ and } \deltat \cong (\deltat \circ \inc_{\mc{S}}) \circ \epsilons.$$
\end{cor} 

\section{Unitizations of $C^*$-Algebras}\label{unitizations} Instead of inducing idempotent monads and comonads on $\ucom$, we could have developed them from scratch much in the same way we developed the idempotent monads and comonads on $\com$. We do exactly that in order to generalize to the noncommutative setting. Definitions~\ref{minunitization}--\ref{mtau} make no reference to whether the $C^*$-algebras are commutative, so they are all applicable in this section. 

\begin{lemma}
    There is a category $\U$ for that objects are triples $(A, B,\alpha)$ where $A$ is a non-unital $C^*$-algebra and $(B,\alpha)$ is a unitization of $A$. A morphism is a pair $(\tau, \pi):(A, B, \alpha) \to (C,D,\beta)$ where $\tau:A \to C$ is a nondegenerate homomorphism, $\pi:B \to D$ is a unital homomorphism, and the diagram
    \[\begin{tikzcd}
        C \arrow{r}{\beta} & D\\
        A \arrow{u}{\tau} \arrow{r}{\alpha} & B \arrow{u}{\pi}
    \end{tikzcd}\]
    commutes. 
\end{lemma}

\begin{proof}
    This proof is identical to the proof in the commutative case, Lemma~\ref{lemucom}. 
\end{proof}

Notice that $\U$ contains $\ucom$ as a full subcategory. 

  For a fixed non-unital $C^*$-algebra, $A$, let $\U_A$ denote the fiber of $\U$ over $A$ where objects have first coordinate $A$ and morphisms have first coordinate $\id_A$. We have that $(A, \tild{A}, \iota)$ is an initial object in $\U_A$ and $(A, M(A), \rho)$ is a final object in $\U_A$.  

\begin{prop}
    There is a functor $\functm: \U \to \U$ given on objects by setting $\functm(A,B,\alpha)$ to be an initial object in $\U_A$. That is, 
    $$\functm(A,B,\alpha) = (A, \tild{A}, \iota).$$
    $\functm$ is given on morphisms $(\tau, \pi):(A, B, \alpha) \to (C, D, \beta)$ by 
    $$\functm(\tau, \pi)= (\tau, \tild{\kappa \circ \tau}): (A, \tild{A}, \iota) \to (C, \tild{C}, \kappa).$$
\end{prop}

\begin{proof}
    This is clear. 
\end{proof}

\begin{prop}
    The assignment $\psi_{(A,B,\alpha)} = (\id_A, \tild{\alpha})$ gives a natural transformation $\psi: \functm \to 1_{\U}$. 
\end{prop}

\begin{proof}
    We show that for all morphisms, $(\tau, \pi): (A, B, \alpha) \to (C, D, \beta)$, the diagram
 \[\begin{tikzcd}[column sep=large]
        (C, \tild{C}, \kappa) \arrow{r}{(\id_C, \tild{\beta})} & (C, D, \beta) \\
        (A, \tild{A}, \iota) \arrow{u}{(\tau, \tild{\kappa \circ \tau})}\arrow{r}{(\id_A, \tild{\alpha})} & (A, B, \alpha) \arrow{u}{(\tau, \pi)}
    \end{tikzcd}\]
    commutes

    This is clear in the first coordinate. For the second, let $\iota(a) \in \tild{A}$. Then we have: 
    \begin{align*}
        \tild{\beta} \circ \tild{\kappa \circ \tau} (\iota(a)) & = \tild{\beta} \circ \kappa \circ \tau(a) & \text{by the commuting square for }(\tau, \tild{\kappa \circ \tau}) \\
        & = \beta \circ \tau (a) & \text{by the commuting square for }(\id_C, \tild{\beta}) \\
        & = \pi \circ \alpha(a) & \text{by the commuting square for $(\tau, \pi)$}\\
        & = \pi \circ \tild{\alpha}(\iota(a)) & \text{by the commuting square for $(\id_A, \tild{\alpha})$}.
    \end{align*}

    $\tild{\beta} \circ \tild{\kappa \circ \tau}$ and $\pi \circ \tild{\alpha}$ are homomorphisms that are nondegenerate on $\iota(A)$ and agree on $\iota(A)$, an essential ideal of $\tild{A}$, therefore must agree on all of $\tild{A}$. 
\end{proof}

\begin{prop}
    $\functm \psi: \functm^2 \to \functm$ and $\psi \functm:\functm^2 \to \functm$ are natural isomorphisms. 
\end{prop}

\begin{proof}
    First, we show $\functm^2 = \functm$. For an object $(A, B, \alpha)$ we have 
    $$\functm^2(A, B, \alpha) = \functm(A, \tild{A}, \iota) = (A, \tild{A}, \iota) = \functm(A, B, \alpha).$$
    For a morphism $(\tau, \pi):(A, B, \alpha) \to (C,D, \beta)$ we have $\functm^2(\tau, \pi):(A, \tild{A}, \iota) \to (B, \tild{B}, \kappa)$ where 
    $$\functm^2(\tau, \pi) = \functm(\tau, \tild{\kappa \circ \tau}) = (\tau, \tild{\kappa \circ \tau}) = \functm(\tau, \pi).$$

    So, $\functm^2=\functm$. Next, consider $\functm \psi:\functm \to \functm$. We have 
    \begin{align*}
        \functm \psi_{(A, B, \alpha)} = \functm(\id_A, \tild{\alpha}) 
        = (\id_A, \tild{\iota \circ \id_A})
        = (\id_A, \id_{\tild{A}}). 
    \end{align*}

    The last equality is because $\tild{\iota \circ \id_A}=\tild{\iota}$ is the unique unital homomoprhism such that $\tild{\iota }\circ \iota = \iota$. Thus, $\tild{\iota \circ \id_A}=\id_{\tild{A}}$. 

    So, $\functm \psi: \functm \to \functm$ is the identity transformation and therefore a natural isomorphism. 

    Similarly, for $\psi \functm: \functm \to \functm$ we have
    \begin{align*}
        \psi \functm(A, B, \alpha) = \psi_{(A, \tild{A}, \iota)} = (\id_A, \tild{\iota})  = (\id_A, \id_{\tild{A}}).
    \end{align*}

    We conclude $\psi \functm: \functm \to \functm$ is also the identity transformation and therefore a natural isomorphism. 
\end{proof}

\begin{definition}
    We denote by $\M$ the full subcategory of $\U$ whose objects $(A, B, \alpha)$ satisfy the property that $\psi_{(A, B, \alpha)}:(A, \tild{A}, \iota) \to (A, B, \alpha)$ is an isomorphism. Let $\mpsi: \U \to \M$ be the functor obtained by regarding $\functm$ as a functor from $\U$ to $\M$ so that $\functm = \inc_{\M} \circ \mpsi$. 
\end{definition}

Notice, $(A, B, \alpha)$ is an object in $\M$ if and only if $\tild{\alpha}: \tild{A} \to B$ is an isomorphism. Thus, we say that $\M$ is the full subcategory of $\U$ containing all objects that are isomorphic to a minimal unitization triple. 

We immediately conclude the following from Lemma~\ref{comonad}. 

\begin{thm}\label{uminthm}
    Let $\functm: \U \to \U$, $\psi: \functm \to 1_\U$, $\M$, and $\mpsi: \U \to \M$ be as above. Then: 
    \begin{enumerate}
        \item[(i)] $(\functm, \psi)$ is an idempotent comonad on $\U$. 
        \item[(ii)] $\inc_\M$ is right adjointable, $\mpsi$ is a right adjoint of $\inc_\M$, and $\psi$ is the counit of the adjunction $\inc_\M \dashv \mpsi$. In particular, $\M$ is a coreflective subcategory of $\U$. 
        \item[(iii)] For all objects $(C,D, \beta)$ in $\U$, $(\functm(C,D,\beta),\psi_{(C,D,\beta)})$ is a universal morphism from $\M$ to $(C,D,\beta)$. 
    \end{enumerate}

    Moreover an object $(A, B, \alpha)$ is in $\M$ if and only if $(A,B,\alpha)$ is isomorphic to a minimal unitization triple. 
\end{thm}

\begin{remark} The diagram for (3) in Theorem~\ref{uminthm} follows. If $(C, D, \beta)$ is an object in $\U$, $(A, B, \alpha)$ is an object in $\M$, and $(\tau, \pi):(A, B, \alpha) \to (C, D, \beta)$, then $(\tau, \tild{\kappa \circ \tau} \circ \tild{\alpha}^{-1}):(A, B, \alpha) \to (C, \tild{C}, \kappa)$ is the unique morphism completing the following diagram:
    \[\begin{tikzcd}
        & (C, D, \beta) \\
        (A, B, \alpha) \arrow{ru}{(\tau,\pi)} \arrow[dashed]{r} & (C, \tild{C}, \kappa) \arrow[swap]{u}{(\id_C, \tild{\beta})} \\
        (A, \tild{A}, \iota) \arrow{u}{(\id_A, \tild{\alpha})} \arrow[swap]{ru}{(\tau, \tild{\kappa \circ \tau})}.
    \end{tikzcd}\]
\end{remark}

\begin{prop}
    There is a functor $\functn: \U \to \U$ given on objects by setting $\functn(A,B, \alpha)$ to be a final object in $\U_A$. That is 
    $$\functn(A,B, \alpha) = (A, M(A), \rho).$$
    $\functn$ is given on morphisms $(\tau, \pi):(A, B, \alpha) \to (C, D, \beta)$ by 
    $$\functn(\tau, \pi) = (\tau, M\tau):(A, M(A), \rho) \to (C, M(C), \sigma).$$ 
\end{prop}

\begin{proof}
    This is clear. 
\end{proof}

\begin{prop}
    The assignment $\eta_{(A,B,\alpha)} = (\id_A, \ov{\alpha})$ gives a natural transformation $\eta: 1_\U \to \functm$. 
\end{prop}

\begin{proof}
    For all morphisms $(\tau, \pi):(A, B, \alpha) \to (C, D, \beta)$ we must check that the diagram 
\[\begin{tikzcd}[column sep=large]
        (C, D, \beta) \arrow{r}{(\id_C, \ov{\beta})} & (C, M(C), \sigma) \\
        (A, B, \alpha) \arrow{u}{(\tau, \pi)} \arrow{r}{(\id_A, \ov{\alpha})} & (A, M(A), \rho) \arrow[swap]{u}{(\tau, M \tau)}
    \end{tikzcd}\]
    commutes.

    This is clear in the first coordinate. For the second, let $\alpha(a) \in \alpha(A)$. Then
    \begin{align*}
        M \tau \circ \ov{\alpha}(\alpha(a)) & = M \tau \circ \rho(a) &\text{by the commuting square for }(\id_A, \ov{\alpha})\\
        & = \sigma \circ \tau (a) &\text{by the commuting square for } (\tau, M \tau) \\
        & = \ov{\beta} \circ \beta \circ \tau (a) &\text{by the commuting square for }(\id_B, \ov{\beta}) \\
        & = \ov{\beta} \circ \pi(\alpha(a)) &\text{by the commuting square for }(\tau, \pi). 
    \end{align*}

    $M \tau \circ \ov{\alpha}$ and $\ov{\beta} \circ \pi$ are homomorphisms that are nondegenerate and agree on $\alpha(A)$, an essential ideal of $B$. Therefore they must agree on all of $B$ and so the diagram commutes. 
\end{proof}

\begin{prop}
    $\functn \eta: \functn \to \functn^2$ and $\eta \functn: \functn \to \functn^2$ are natural isomorphisms. 
\end{prop}

\begin{proof}
    First, we show $\functn^2=\functn$. For an object $(A, B, \alpha)$ in $\U$ we have 
    $$\functn^2(A, B, \alpha) = \functn(A, M(A), \rho) = (A, M(A), \rho) = \functn(A, B, \alpha).$$
    For a morphism $(\tau, \pi):(A, B, \alpha) \to (C,D, \beta)$ we have $\functn^2(\tau, \pi):(A, M(A), \rho) \to (C, M(C), \sigma)$ where
    $$\functn^2(\tau, \pi) = \functn(\tau, M \tau) = (\tau, M \tau) = \functn(\tau, \pi).$$

    So, $\functn^2 = \functn$. Next, we consider $\functn \eta:\functn \to \functn$. We have
    $$\functn \eta_{(A, B, \alpha)} =\functn(\id_A, \ov{\alpha}) = (\id_A, M (\id_A)) = (\id_A, \id_{M(A)}).$$

    The last equality is true because $M(\id_A)$ is the unique morphism such that $M(\id_A) \circ \rho = \rho$. Thus, $M(\id_A) = \id_{M(A)}$. 

    For $\eta \functn: \functn \to \functn$ we have 
    $$\eta \functn(A, B, \alpha) = \eta_{(A, M(A), \rho)} = (\id_A, \ov{\rho}) = (\id_A, \id_{M(A)}).$$

    The last equality is true because $\ov{\rho}$ is the unique morphism such that $\ov{\rho} \circ \rho = \rho$ so $\ov{\rho} = \id_{M(A)}$.
\end{proof}

\begin{definition}
    We denote by $\N$ the full subcategory of $\U$ whose objects $(A, B, \alpha)$ satisfy the property that $\eta_{(A, B, \alpha)}:(A, B, \alpha) \to (A, M(A), \rho)$ is an isomorphism. Let $\neta: \U \to \N$ be the functor obtained by regarding $\functn: \U \to \U$ as a functor from $\U$ to $\N$. So, $\functn = \inc_{\N} \circ \neta$. 
\end{definition}

Notice, $(A, B, \alpha)$ is an object in $ \N$ if and only if $\ov{\alpha}: B \to M(A)$ is an isomorphism. So, we say $\N$ is the full subcategory of objects in $\U$ isomorphic to a maximal unitization triple. 

We immediately conclude the following from Lemma~\ref{monad}. 

\begin{thm}\label{umaxthm}
    Let $\functn: \U \to \U$, $\eta:1_\U \to \functn$, $\N$, and $\neta:\U \to \N$ be as defined above. Then: 
    \begin{enumerate}
        \item[(i)] $(\functn, \eta)$ is an idempotent monad on $\U$. 
        \item[(ii)] $\inc_\N$ is left adjointable, $\neta$ is a left adjoint of $\inc_\N$, and $\eta$ is the unit of the adunction $\neta \dashv \inc_\N$. In particular, $\N$ is a reflective subcategory of $\U$. 
        \item[(iii)] For all objects $(A,B, \alpha)$ in $\U$, $(\functn(A,B,\alpha),\eta_{(A,B,\alpha)})$ is a universal morphism from $(A, B, \alpha)$ to $\N$. 
        \end{enumerate}
        Moreover, an object $(A,B, \alpha)$ is in $\N$ if and only if $(A, B,\alpha)$ is isomorphic to a maximal unitization triple. 
\end{thm}

\begin{remark}
    The diagram for item (3) in Theorem~\ref{umaxthm} is as follows. For an object $(A, B, \alpha)$ in $\U$, $(C, D, \beta)$ an object in $\N$, and $(\tau, \pi):(A, B, \alpha) \to (C,D, \beta)$, $(\tau, \ov{\beta}^{-1} \circ M \tau)$ is the unique morphism completing the following diagram:
\end{remark}

\[\begin{tikzcd}
    & (C, M(C), \sigma) \\
    (A, M(A), \rho) \arrow{ru}{(\tau, M \tau)} \arrow[dashed]{r} & (C,D, \beta) \arrow[swap]{u}{(\id_C, \ov{\beta})}\\
    (A, B, \alpha) \arrow{u}{(\id_A, \ov{\alpha})} \arrow[swap]{ru}{(\tau, \pi)}.
\end{tikzcd}\]

\begin{prop}\label{minmax}
    The natural transformations $\functm \eta:\functm \to \functm \functn$ and $\functn \psi: \functn \to \functn \functm$ are natural isomorphisms. 
\end{prop}

\begin{proof}
    For an object $(A, B, \alpha)$ we have 
    \begin{align*}
        \functm \eta_{(A, B, \alpha)}  = \functm(\id_A, \tild{\alpha}) = (\id_A, \tild{\iota \circ \id_A}) = (\id_A, \id_{\tild{A}}). 
    \end{align*}

    So, $\functm \eta$ is a natural isomorphism. Similarly, 
    \begin{align*}
        \functn \psi_{(A, B, \alpha)}  = \functn(\id_A, \ov{\alpha})= (\id_A, M(\rho \circ \id_A))= (\id_A, \id_{M(A)}).
    \end{align*}

    So, $\functn \psi$ is also a natural isomorphism. 
\end{proof}

An immediate consequence of Proposition~\ref{minmax} is the following. 

\begin{thm}
    The categories $\N$ and $\M$ are equivalent. The adjunction $\neta \circ \inc_{\M} \dashv \mpsi \circ \inc_{\N}$ is an adjoint equivalence between $\N$ and $\M$. Moreover, we have 
    $$\neta \cong (\neta \circ \inc_{\M}) \circ \mpsi \text{ and } \mpsi \cong (\mpsi \circ \inc_{\N}) \circ \neta.$$
\end{thm} 

\section{Conclusion}

Notice that constructing the idempotent monad and idempotent comonad on $\U$ in Section~\ref{unitizations} does not rely on any of the results presented in Section~\ref{compactifications} or~\ref{ucom}. We could have just as well presented the maximal-normal equivalence in the category of unitizations of $C^*$-algebras first. This is our most general example. We could then specialize to the commutative setting and, using the contravariant equivalence, induce the pair of subcategories satisfying the maximal-normal relationship in $\com$. That is, we may present this story in reverse order. 

We opted to present the maximal-normal equivalence in the category of compactifications first. This example has the benefit of being more conceptually accessible. 

As mentioned in the introduction and in~\cite{BKQ25} this example is one of many examples of subcategories satisfying the maximal-normal equivalence in categories involving $C^*$-algebras. Moreover, our example in the topological setting indicates that these pairs are not unique to $C^*$-theory and are common in other areas as well.

\end{document}